\documentclass{amsart}
\usepackage{amsmath}
\usepackage{amssymb}
\usepackage[all]{xy}
\usepackage{graphicx}
\usepackage{mathrsfs}
\usepackage{stmaryrd}
\usepackage[active]{srcltx}

\usepackage{color}

\RequirePackage{ifpdf}
\ifpdf
  \IfFileExists{pdfsync.sty}{\RequirePackage{pdfsync}}{}
  \RequirePackage[pdftex,
   colorlinks = true,
   pagebackref,
   bookmarksopen=true]{hyperref}
\else
  \RequirePackage[hypertex]{hyperref}
\fi

\usepackage{tikz}
\tikzset{
	MyPersp/.style={scale=2,x={(0.8cm,0cm)},y={(0cm,0.25cm)},
    z={(0cm,1cm)}},
	MyPoints/.style={fill=white,draw=black,thick}
		}

\RequirePackage{ae, aecompl, aeguill} 

    \def\CM{{\mathbb{C}}}

  \def\hg{{\mathfrak h}}

    \def\PM{{\mathbb{P}}}
    \def\QM{{\mathbb{Q}}}
    \def\RM{{\mathbb{R}}}

    \def\ZM{{\mathbb{Z}}}

\def\a{\alpha}
\def\b{\beta}
\def\g{\gamma}

\def\e{\varepsilon}

\def\l{\lambda}

\def\z{\zeta}

\newtheorem{thm*}{Theorem}
\numberwithin{equation}{section}
\newtheorem{thm}{Theorem}[section]

\newtheorem{prop}[thm]{Proposition}

\theoremstyle{definition}

\theoremstyle{remark}
\newtheorem{rmk}[thm]{Remark}
\newtheorem{ex}[thm]{Example}

\DeclareMathOperator{\Lie}{Lie}

\DeclareMathOperator{\Poinc}{Poinc}
\DeclareMathOperator{\id}{id}
\DeclareMathOperator{\std}{std}

\def\top{{\mathrm{top}}}
\newcommand{\un}{\underline}
\DeclareMathOperator{\Tr}{Tr}

\newcommand{\simto}{\stackrel{\sim}{\longrightarrow}}

\title[Kazhdan-Lusztig conjectures]{Kazhdan-Lusztig conjectures and shadows of Hodge
theory}

\author{Ben Elias}
\address{Department of Mathematics, University of Oregon, Eugene OR, USA.}
\email{belias@uoregon.edu}

\author{Geordie Williamson}
\address{Max-Planck-Institut f\"ur Mathematik, Vivatsgasse 7, 53111,
  Bonn, Germany.
}
\email{geordie@mpim-bonn.mpg.de}

\begin{document}

\begin{abstract}
We give an informal introduction to the authors' work on some
conjectures of Kazhdan and Lusztig, building on work of
Soergel and de Cataldo-Migliorini. This article is an expanded version
of a lecture given by the second author at the Arbeitstagung in memory
of Frederich Hirzebruch.
\end{abstract}

\maketitle

\section{Introduction}

It was a surprise and honour to be able to speak
about our recent work at the Arbeitstagung in memory of
Hirzebruch. These feelings are heightened by the fact that the decisive moments in the development of
our joint work occurred at the Max-Planck-Institut in Bonn, which owes its very existence to
Hirzebruch. In the following introduction we have tried to emphasize the aspects of
our work which we believe Hirzebruch would have most enjoyed: compact
Lie groups and the topology of their homogenous spaces; characteristic
classes; Hodge theory; and more generally the remarkable topological properties of projective
algebraic varieties.

Let $G$ be a connected compact Lie group and $T$ a maximal torus. A fundamental
object in mathematics is the flag manifold $G/T$. We briefly recall Borel's
beautiful and canonical description of its cohomology. Given a character $\l : T \to
\CM^*$ we can form the line bundle
\[
L_\l := G \times_T \CM
\]
on $G/T$, defined as the quotient of $G \times \CM$ by $T$-action
given by $t \cdot (g, x) := (gt^{-1}, \lambda(t)x)$. Taking the Chern
class of $L_\l$ yields a homomorphism
\[
X(T) \to H^2(G/T) : \lambda \mapsto c_1(L_\l).
\]
from the lattice of characters to the second cohomology of
$G/T$. If we identify $X(T) \otimes_\ZM \RM = (\Lie T)^*$ via the
differential and extend multiplicatively we get a morphism of graded
algebras
\[
R := S((\Lie T)^*) \to H^\bullet(G/T; \RM).
\] 
called the \emph{Borel homomorphism}. (We let $R$ denote
the symmetric algebra on the dual of $\Lie T$.) Borel showed that his
homomorphism is surjective and identified its kernel with the ideal
generated by $W$-invariant polynomials of positive degree. Here $W =
N_G(T)/T$ denotes the Weyl group of $G$ which acts on $T$ by conjugation,
hence on $\Lie T$ and hence on $R$.

For example, let $G = U(n)$ be the unitary group, and $T$ the subgroup
of diagonal matrices ($\cong (S^1)^n$). Then the coordinate functions give an
identification $R = \RM[x_1, \dots, x_n]$, and $W$ is the symmetric group
on $n$-letters acting on $R$ via permutation of variables. The
Borel homomorphism gives an identification
\[
\RM[x_1, \dots, x_n]/\langle e_i \;|\; 1 \le i \le n \rangle  = H^\bullet(G/T;\RM) 
\]
where $e_i$ denotes the $i^{th}$ elementary symmetric polynomial in
$x_1, \dots, x_n$.

Let $G_\CM$ denote the complexification of $G$ and choose a Borel
subgroup $B$ containing the complexification of $T$. (For example
if $G = U(n)$ then $G_\CM = GL_n(\CM)$ and we could take $B$ to be the
subgroup of upper-triangular matrices.) A fundamental fact is that the
natural map
\[
G/T \to G_\CM /B
\]
is a diffeomorphism, and $G_\CM/B$ is a projective
algebraic variety.

For example, if $G = SU(2) \cong S^3$ then $G/T = S^2$ is the
base of the Hopf fibration, and the above diffeomorphism is
$S^2 \simto \PM^1\CM$. More generally for $G = U(n)$ the above
diffeomorphism can be seen as an instance of Gram-Schmidt
orthogonalization. Fix a Hermitian form on $\CM^n$. Then $G_\CM/B$ parametrizes complete flags on $\CM^n$, while $G/T$ parametrizes collections of $n$ ordered orthogonal complex lines. These spaces are clearly isomorphic.

The fact that $G/T = G_\CM/B$ is a projective algebraic variety means that its cohomology satisfies a number of deep theorems from complex algebraic geometry. Set $H = H^\bullet(G_\CM/B;
\RM)$ and let $N$ denote the complex dimension of $G_\CM/B$. For us the following two results (the ``shadows of Hodge theory'' of the title) will be of fundamental importance.

\begin{thm}[Hard Lefschetz theorem] \label{hL} Let $\lambda \in H^2$
  denote the Chern class of
  an ample line bundle on $G_\CM/B$ (i.e. $\lambda \in (\Lie T)^*$ is
  a `dominant
  weight', see \eqref{eq:ample cone}). Then for all $0 \le i \le N$
  multiplication by $\lambda^{N-i}$ gives an isomorphism:
\[
\l^{N-i} : H^i \simto H^{2N-i}.
\]  
\end{thm}


Because $G/T$ is a compact manifold, Poincar\'e duality states that $H^i$ and $H^{2N-i}$ are non-degenerately paired
by the Poincar\'e pairing $\langle -, - \rangle_{\Poinc}$. On the other hand, after fixing
$\lambda$ as above the hard Lefschetz theorem gives us a way of
identifying $H^i$ and $H^{2N-i}$. The upshot is that for $0 \le i \le N$ we obtain a non-degenerate
\emph{Lefschetz form}:
\begin{align*}
H^i \times H^i &\to \RM \\
(\alpha, \beta) & \mapsto \langle \alpha,
\lambda^{N-i} \beta \rangle_{\Poinc}.
\end{align*}
On the middle dimensional cohomology the Lefschetz form is just the
Poincar\'e pairing. This is the only Lefschetz form which
  does not depend on the choice of ample class $\lambda$.

\begin{thm}[Hodge-Riemann bilinear relations] \label{HR} For $0 \le i \le N$ the
  restriction of the Lefschetz form to $P^i := \ker
  (\lambda^{N-i+1}) \subset H^i$ is $(-1)^{i/2}$-definite.
\end{thm} 

Some comments are in order:
\begin{enumerate}
\item The odd cohomology of $G/T$ vanishes as can be seen, for example, from the surjectivity of the Borel homomorphism. Hence the
  sign $(-1)^{i/2}$ makes sense.
\item For an arbitrary smooth projective algebraic variety the
  Hodge-Riemann bilinear relations are more complicated, involving the
  Hodge decomposition and a Hermitian form on the complex cohomology
  groups. However, the cohomology of the flag variety is always in $(p,p)$-type, so that we may use the simpler formulation above.
\item We will not make it explicit, but the Hodge-Riemann bilinear relations give formulas for the signatures of all Lefschetz forms in terms of the graded dimension of $H$.
\end{enumerate}

We now come to the punchline of this survey. The hard Lefschetz theorem and Hodge-Riemann bilinear relations for $H^\bullet(G/B; \RM)$ are deep consequences of Hodge theory.
On the other hand, we have seen that the Borel homomorphism gives us
an elementary description of $H^\bullet(G/B; \RM)$ in terms of
commutative algebra and invariant theory. Can one establish
the hard Lefschetz theorem and Hodge-Riemann bilinear relations for $H^\bullet(G/B;\RM)$ algebraically? A crucial motivation for this question is the fact that $H^\bullet(G/B;\RM)$ has
various algebraic cousins (described in \S\ref{arbitrary}) for which no geometric description is known. Remarkably, these cousins still satisfy analogs of Theorems \ref{hL} and \ref{HR}.
Establishing these Hodge-theoretic properties algebraically is the cornerstone of the authors' approach to conjectures of Kazhdan-Lusztig and Soergel.

The structure of this (very informal) survey is as follows. In
\S\ref{ic} we give a lightning introduction to intersection
cohomology, which provides an improved cohomology theory for singular
algebraic varieties. In \S\ref{Schubert} we discuss Schubert
varieties, certain (usually singular) subvarieties of the flag variety
which play an important role in representation theory. We also discuss
Bott-Samelson resolutions of Schubert varieties. In \S\ref{Soergel} we
discuss Soergel modules. The point is that one can give a purely
algebraic/combinatorial description of the intersection cohomology of
Schubert varieties, which only depends on the underlying Weyl group.
In \S\ref{arbitrary} we discuss Soergel modules
for arbitrary Coxeter groups, which (currently) have no geometric
interpretation. We also state our main theorem that these
modules satisfy the ``shadows of Hodge theory''. Finally, in
\S\ref{dihedral} we discuss the amusing example of the coinvariant ring of a
finite dihedral group.

\section{Intersection cohomology and the decomposition theorem} \label{ic}

Poincar\'e duality, the hard Lefschetz theorem and Hodge-Riemann
bilinear relations hold for the cohomology of any smooth projective
variety. The statements of these results usually fail for singular
varieties. However, in the 1970s Goresky and MacPherson invented
intersection cohomology \cite{GM1,GM2} and it was later proven that the analogues of
these theorems hold for intesection cohomology. In this section we will try to give the
vaguest of vague ideas as to what is going on, 
and hopefully convince the reader to go and read more. (The authors'
favourite introduction to the theory is \cite{dCM0} whose emphasis agrees
largely with that of this survey.\footnote{Due, no doubt, to the
  influence which their work has had on the authors.}
 More information is contained in \cite{Borel,Rietsch,Arabia} with the bible being \cite{BBD}. To stay motivated, Kleiman's excellent
history of the subject \cite{Kleiman} is a must.)


Intersection cohomology associates to any complex variety $X$ its
``intersection cohomology groups'' $IH^\bullet(X)$ (throughout this
article we always take coefficients in $\RM$, however there are
versions of the theory with $\QM$ and $\ZM$-coefficients). Here are
some basic properties of intersection cohomology:
\begin{enumerate}
\item $IH^\bullet(X)$ is a graded vector space, concentrated in
  degrees between 0 and $2N$, where $N$ is the complex dimension of
  $X$;
\item if $X$ is smooth then $IH^\bullet(X) = H^\bullet(X)$;
\item if $X$ is projective then $IH^\bullet(X)$ is equipped with a
  non-degenerate Poincar\'e pairing $\langle -, - \rangle_{Poinc}$,
  which is the usual Poincar\'e pairing for $X$ smooth.
\end{enumerate}
However we caution the reader that:
\begin{enumerate}
\item the assignment $X \mapsto IH^\bullet(X)$ is not functorial: in general a
  morphism $f : X \to Y$ does not induce a pull-back map on intersection
  cohomology;
\item $IH^\bullet(X)$ is not a ring, but rather a module over the
  cohomology ring $H^\bullet(X)$.
\end{enumerate}
(These two ``failings'' become less worrying when one interprets
intersection cohomology in the language of constructible sheaves.) 
Finally, we come to the two key properties that will concern us in
this article. We assume that $X$ is a
projective variety (not necessarily smooth):
\begin{enumerate}
\item multiplication by the first Chern class of an ample line bundle
  on $IH^\bullet(X)$ satisfies the hard Lefschetz theorem;
\item the groups $IH^\bullet(X)$ satisfy the Hodge-Riemann bilinear relations.
\end{enumerate}
(To make sense of this second statement, one needs to know that
$IH^\bullet(X)$ has a Hodge decomposition. This is true, but we will
not discuss it. Below, we will only consider varieties whose Hodge
decomposition only involves components of type $(p,p)$ and so the
naive formulation of the Hodge-Riemann bilinear relations in the form
of Theorem \ref{HR} will be sufficient.)

\begin{ex} \label{ex1}
Consider the
Grassmannian $Gr(2,4)$ of planes in $\CM^4$.
It is a smooth projective algebraic variety of complex dimension $4$.
Let $0 \subset \CM
\subset \CM^2 \subset \CM^3 \subset \CM^4$ denote the standard
coordinate flag on $\CM^4$.
For any sequence of natural numbers $\un{a} := (0 = a_0 \le a_1 \le a_2 \le a_3
\le a_4 = 2)$ satisfying $a_i \le a_{i+1} \le a_i + 1$, consider the subvariety
\[
C_{\un{a}} := \{ V \in Gr(2,4) \; | \; \dim(V \cap \CM^i)  = a_i \}.
\]
It is not difficult (by writing down charts for the Grassmannian)
to see that each $C_{\un{a}}$ is isomorphic to $\CM^{d(\un{a})}$ where
$d(\un{a}) = 7 - \sum_{i=0}^4 a_i$. Hence
$Gr(2,4)$ has a cell-decomposition with cells of real dimension
$0,2,4,4,6,8$. The cohomology $H^\bullet(Gr(2,4))$ is as
follows:
\begin{equation*}
\begin{array}{|c|c|c|c|c|c|c|c|c|}
 \hline 
0& 1 & 2 & 3 & 4 & 5 & 6 & 7 & 8 \\
 \hline 
\RM & 0 & \RM & 0& \RM^{2} & 0 & \RM & 0 & \RM \\
\hline
\end{array}
\end{equation*}
It is an easy exercise to use Schubert calculus (see
e.g. \cite[III.3]{Hiller}, which also discusses $Gr(2,4)$ in more detail) to check the hard
Lefschetz theorem and Hodge-Riemann bilinear relations by hand.

Now consider the subvariety
\[
X := \{ V \in Gr(2,4) \; | \; \dim (V \cap \CM^2) \ge 1 \}.
\]
Then $X$ coincides with the closure of the cell $C_{0 \le 0 \le 1 \le
  1 \le 2} \subset Gr(2,4)$ (and thus is an example of a ``Schubert
variety'', as we will discuss in the next section).
Hence $X$ has real dimension 6 and has a cell-decomposition with cells of dimension
$(0,2,4,4,6)$. Its cohomology is as follows:
\begin{equation*}
\begin{array}{|c|c|c|c|c|c|c|}
 \hline 
0& 1 & 2 & 3 & 4 & 5 & 6 \\
 \hline 
\RM & 0 & \RM & 0& \RM^{ 2} & 0 & \RM \\
\hline
\end{array}
\end{equation*}
We conclude that $X$ cannot satisfy Poincar\'e duality or the hard
Lefschetz theorem. In particular $X$ must be
singular. In fact, $X$ has a unique singular point $V_0 =
\CM^2$. We will see below that the intersection cohomology
$IH^\bullet(X)$ is as follows:
\begin{equation*}
\begin{array}{|c|c|c|c|c|c|c|}
 \hline 
0& 1 & 2 & 3 & 4 & 5 & 6 \\
 \hline 
\RM & 0 & \RM^{ 2} & 0& \RM^{ 2} & 0 & \RM \\
\hline
\end{array}
\end{equation*}
So in this example $IH^\bullet(X)$ seems to fit the bill (at least on
the level of Betti numbers) of
rescuing Poincar\'e duality and the hard Lefschetz theorem in a ``minimal'' 
way.
\end{ex}

Probably the most fundamental theorem about intersection cohomology is
the decomposition theorem. In its simplest form it says the following:

\begin{thm}[Decomposition theorem \cite{BBD, Saito, dCM, dCM2}] Let $f : \widetilde{X} \to X$ be a resolution,
  i.e. $\widetilde{X}$ is smooth and $f$ is a projective birational
  morphism of algebraic varieties. Then $IH^\bullet(X)$ 
  is a direct summand of $H^\bullet(\widetilde{X})$, as modules over
  $H^\bullet(X)$.\end{thm}

The decomposition theorem provides an invaluable tool for calculating
intersection cohomology, which is otherwise a very difficult task.

\begin{ex} In Example \ref{ex1} we discussed the variety
\[
X := \{ V \in Gr(2,4) \; | \; \dim (V \cap \CM^2) \ge 1 \}
\]
which is projective with unique singular point $V_0 = \CM^2$.
Now $X$ has a natural resolution $f : \widetilde{X} \to X$ where
\[
\widetilde{X} = \{ (V, W) \in Gr(2,4) \times \PM(\CM^2) \; | \; W
\subset V \cap \CM^2 \}
\]
and $f(V,W) = V$. Clearly $f$ is an isomorphism over $X \setminus \{
V_0 \}$ and has fibre $\PM^1= \PM(\CM^2)$ over the singular point $V_0$. Also, the projection $(V,W)
\mapsto W$ realizes $\widetilde{X}$ as a $\PM^2$-bundle over
$\PM^1$. In particular, $\widetilde{X}$ is smooth and its cohomology
is as follows:
\begin{equation*}
\begin{array}{|c|c|c|c|c|c|c|}
 \hline 
0& 1 & 2 & 3 & 4 & 5 & 6 \\
 \hline 
\RM & 0 & \RM^{ 2} & 0& \RM^{ 2} & 0 & \RM \\
\hline
\end{array}
\end{equation*}

We conclude
by the decomposition theorem that $IH^\bullet(X)$ is a summand of
$H^\bullet(\widetilde{X})$. In this case one has equality:
$IH^\bullet(X) = H^\bullet(\widetilde{X})$. One can see this directly
as follows: first one checks that the pull-back map
$H^i(X) \to H^i(\widetilde{X})$ is injective.
Now, because $IH^\bullet(X)$ is an $H^\bullet(X)$-stable
summand of $H^\bullet(\widetilde{X})$ containing $\RM =
H^0(\widetilde{X})$ we conclude that $IH^i(X)  = H^i(\widetilde{X})$
for $i \ne 2$. Finally, we must have $IH^2(X) =  H^2(\widetilde{X})$
because $IH^\bullet(X)$ satisfies Poincar\'e duality.

Let us now discuss the hard Lefschetz theorem and Hodge-Riemann
bilinear relations for $IH^\bullet(X)$. Let $\lambda$ be the class of an ample line bundle on $X$. Because $IH^\bullet(X) =
H^\bullet(\widetilde{X})$ in this example, the action of $\lambda$ on $IH^\bullet(X)$ is identified with the action of $f^*\lambda$ on $H^\bullet(\widetilde{X})$.
We would like to know that $f^*\lambda$ acting on $H^\bullet(\widetilde{X})$ satisfies the
the hard Lefschetz theorem and Hodge-Riemann bilinear relations
\emph{even though $f^*\lambda$ is not an ample class on
  $\widetilde{X}$}. This simple observation is the starting point for
beautiful work of de Cataldo and Migliorini \cite{dCM,dCM2}, who give a Hodge-theoretic proof of the decomposition theorem.
\end{ex}

\section{Schubert varieties and Bott-Samelson resolutions} \label{Schubert}

Recall our connected compact Lie group $G$, its complexification $G_\CM$, the
maximal torus $T \subset G$ and the Borel subgroup $T \subset B
\subset G_\CM$. To $(G,T)$ we may associate a root system $\Phi \subset (\Lie
T)^*$. Our choice of Borel subgroup is equivalent to a choice of
simple roots $\Delta \subset \Phi$. As we discussed in the
introduction, the Weyl group $W = N_G(T) / T$ acts on $\Lie T$ as a
reflection group. The choice of simple roots $\Delta \subset \Phi$
gives a choice of \emph{simple reflections} $S \subset W$. These simple
reflections generate $W$ and with respect to these generators $W$ admits a
\emph{Coxeter presentation}:
\[
W = \langle s \in S \; | s^2 = \id, (st)^{m_{st}} = \id \rangle
\]
where $m_{st} \in \{ 2, 3, 4,6 \}$ can be read off the Dynkin
diagram of $G$. Given $w \in W$ a \emph{reduced expression} for $w$ is
an expression $w = s_1 \dots s_m$ with $s_i \in S$, having shortest length
amongst all such expressions. The \emph{length} $\ell(w)$ of $w$ is the
length of a reduced expression. The Weyl group $W$
is finite, with a unique longest element $w_0$.

From now on we will work with the flag variety $G_\CM/ B$ in its
incarnation as a projective algebraic variety. It is an important fact
(the ``Bruhat decomposition'') that $B$ has finitely many orbits on
$G_\CM/B$
which are parametrized by the Weyl group $W$. In formulas we write:
\[
G_\CM/B = \bigsqcup_{w \in W} B \cdot wB/B
\]
Each $B$-orbit $B \cdot wB/B$ is isomorphic to an affine space and its
closure
\[
X_w := \overline{B \cdot wB/B}
\]
is a projective variety called a \emph{Schubert variety}. It is of complex dimension
$\ell(w)$. The two extreme cases are $X_{\id}= B/B$, a point, and
$X_{w_0} = G_\CM / B$, the full flag variety.

More generally, given any subset $I \subset S$ we have a parabolic
subgroup $B \subset P_I \subset G$ generated by $B$ and (any choice of
representatives of) the subset $I$. The quotient $G/P_I$ is also a
projective algebraic variety (called a \emph{partial flag variety})
and the Bruhat decomposition takes the
form
\[
G/P_I := \bigsqcup_{w \in W^I} B \cdot wB/P_I
\]
where $W^I$ denotes a set of minimal length representatives for the
cosets $W/W_I$. Again, the Schubert varieties are the closures $X_w^I
:= \overline{B \cdot wB/P_I} \subset G/P_I$, which are projective
algebraic varieties of dimension $\ell(w)$.

\begin{ex} We discussed the more general setting of $G/P_I$ to make contact with
  the Grassmannian in Example \ref{ex1}. Indeed, $Gr(2,4) \cong
  GL_4(\CM)/P$ where $P$ is the stabilizer of the fixed coordinate subspace $\CM^2
  \subset \CM^4$. If $B$ denotes the stabilizer of the coordinate flag
  $0 \subset \CM^1 \subset \CM^2 \subset \CM^3 \subset \CM^4$ (the upper triangular matrices) then the cells $C_{\un{a}}$ of
  Example \ref{ex1} are $B$-orbits on $Gr(2,4)$. Hence our $X$ is an example of a singular Schubert variety.

\end{ex}

Schubert varieties are rarely smooth. We now discuss how to
construct resolutions. We will focus on Schubert varieties in the full flag
variety, although similar constructions work for Schubert varieties in
partial flag varieties. Choose $w \in W$ and fix a reduced expression $w = s_1 s_2 \dots
s_m$. For any $1 \le i \le m$ let us alter our notation and write
$P_i$ for $P_{\{s_i\}} = \overline{Bs_iB}$, a (minimal) parabolic
subgroup associated to the reflection $s_i$. Consider the space
\[
BS(s_1, \dots, s_m ) := P_1 \times^B P_2 \times^B \dots \times^B P_m/B.
\]
The notation $\times^B$ indicates that $BS(s_1, \dots, s_m)$ is the quotient of $P_1 \times P_2 \times \dots \times
P_m$ by the action of $B^m$ via
\[
(b_1, b_2, \dots, b_m ) \cdot (p_1, \dots, p_m) = (p_1b_1^{-1},
b_1p_2b_2^{-1}, \dots, b_{m-1}p_mb_{m}^{-1}).
\]
Then $BS(s_1, \dots, s_m)$ is a smooth projective \emph{Bott-Samelson
variety} and the
multiplication map $P_1 \times \dots \times P_m \to G$ induces a
morphism
\[
f : BS(s_1, \dots, s_m) \to X_w
\]
which is a resolution of $X_w$. (See \cite{DemBS,Hansen} and \cite[\S 2]{Brion} for further
discussion and applications of Bott-Samelson resolutions. The name
Bott-Samelson resolution comes from \cite{BS} where related spaces are considered in
the context of loop spaces of compact Lie groups.)

\begin{ex} \label{ex:bs}
  If $G_\CM = GL_n$, Bott-Samelson resolutions admit a
  more explicit description. Recall that $GL_n/B$ is the variety of
  flags $V_\bullet = (0 = V_0 \subset V_1 \subset V_2 \subset \dots \subset
  V_n = \CM^n)$ with
  $\dim V_i = i$. We identify $W$ with the symmetric
  group $S_n$ and $S$ with the set of simple transpositions $\{ s_i =
  (i,i+1) \; | \; 1 \le i \le n-1 \}$. Given a reduced expression
  $s_{i_1} \dots s_{i_m}$ for $w \in W$ consider the variety
  $\widetilde{BS}(s_{i_1}, \dots , s_{i_m})$ of all
  $m$-tuples of flags $(V_\bullet^a)_{0 \le a \le m}$ such that:
  \begin{enumerate}
  \item  $V^0_\bullet$ is the coordinate flag $V_\bullet^{\std{}} = (0 \subset \CM^1 \subset
    \dots \subset \CM^n)$;
  \item for all $1 \le a \le m$, $V^a_j = V_j^{a-1}$ for $j \ne i_a$. 
  \end{enumerate}
 That is, 
  $\widetilde{BS}(s_{i_1}, \dots , s_{i_m})$ is the variety of 
sequences of $m+1$ flags which begin at the coordinate flag, and
where, in passing from the $(j-1)^{\textrm{st}}$ to the $j^{\textrm{th}}$ step, we are
only allowed to change the  $i_j^{\textrm{th}}$ dimensional subspace.

Let $p_0 = 1$. Then the map 
\[
(p_1, \dots, p_m) \mapsto (p_0 \dots p_a V_\bullet^{\std{}})_{a = 0}^m
\]
gives an isomorphism $BS(s_1, \dots,
s_m) \to \widetilde{BS}(s_1, \dots, s_m)$. Under this isomorphism the
map $f$ becomes the projection to the final flag: $f((V_\bullet^a)_{a
  = 1}^m) = V_\bullet^m$.
\end{ex}

\section{Soergel modules and intersection cohomology} \label{Soergel}

In a landmark paper \cite{So}, Soergel explained how to calculate the
intersection cohomology of Schubert varieties in a purely algebraic
way. Though much less explicit, one way of viewing this result is as a generalization of Borel's
description of the cohomology of the flag variety. 

The idea is as follows. In the last section we discussed the
Bott-Samelson resolutions of Schubert varieties
\[
f : BS(s_1, \dots, s_m) \to X_w \subset G_{\CM}/B
\]
where $w = s_1 \dots s_m$ is a reduced expression for $w$. By the
decomposition theorem $IH^\bullet(X_w)$, the intersection cohomology
of the Schubert variety $X_w \subset G_{\CM}/B$, is a summand of
$H^\bullet(BS(s_1, \dots, s_m))$. Moreover, we have pull-back maps
\[
H^\bullet(G_\CM/B) \twoheadrightarrow H^\bullet(X_w) \to H^\bullet(BS(s_1, \dots, s_m))
\]
and $IH^\bullet(X_w)$ is even a summand of $H^\bullet(BS(s_1, \dots,
s_m))$ as an $H^\bullet(G_\CM/B)$-module. (The surjectivity of the
restriction map $H^\bullet(G_\CM/B) \twoheadrightarrow H^\bullet(X_w)$
follows because both spaces have compatible cell-decompositions.)
Remarkably, this algebraic
structure already determines the summand $IH^\bullet(X_w)$ (see
\cite[Erweiterungssatz]{So}): 

\begin{thm}[Soergel]\label{thm:s1}
Let $w = s_1 \dots s_m$
  denote a reduced expression for $w$ as above.
Consider $H^\bullet(BS(s_1, \dots, s_m))$ as a
  $H^\bullet(G_\CM/B)$-module. Then $IH(X_w)$ may be described as the
  indecomposable graded $H^\bullet(G_\CM/B)$-module direct summand
  with non-trivial degree zero part.
\end{thm}

A word of caution: The realization of $IH^\bullet(X_w)$ inside
$H^\bullet(BS(s_1, \dots, s_m))$ is not canonical in
general. We can certainly decompose $H^\bullet(BS(s_1, \dots, s_m))$ into graded
indecomposable $H^\bullet(G_\CM/B)$-modules. Although this
decomposition is not canonical, the Krull-Schmidt
theorem ensures that the isomorphism type and multiplicities of
indecomposable summands do not depend on the chosen decomposition. The
above theorem states that, for any such decomposition, the unique
indecomposable module with non-trivial degree zero part is isomorphic to
$IH^\bullet(X_w)$ (as an $H^\bullet(G_\CM/B)$-module).

We now explain (following Soergel) how one may give an algebraic
description of all players in the above theorem. Recall that $R = S((\Lie T)^*)$ denotes the symmetric algebra on the
dual of $\Lie T$, graded so that $(\Lie T)^*$ has degree 2. The
Weyl group $W$ acts on $R$, and for any simple reflection $s \in S$ we
denote by $R^s$ the invariants under $s$. It is not difficult to see
that $R$ is a free graded module of rank $2$ over $R^s$ with basis
$\{1,\alpha_s\}$, where $\alpha_s$ is the simple root associated to $s
\in S$. (In essence this is the high-school fact that any polynomial
can be written as the sum of its even and odd parts.)

The starting point is the following observation:

\begin{prop}[Soergel] \label{bs}
One has an isomorphism of graded algebras
\[
H^\bullet(BS(s_1, \dots, s_m)) = R \otimes_{R^{s_1}} R \otimes_{R^{s_2}} \dots
R \otimes_{R^{s_m}} R \otimes_R\RM
\]
where the final term is an $R$-algebra via $\RM \cong R/R^{>0}$.
\end{prop}

For example, for any $s \in S$ we have $BS(s) = P_s/B \cong \PM^1$ and
$R \otimes_{R^s}R \otimes_R \RM = R \otimes_{R^s} \RM$ is
2-dimensional, with graded basis $\{ 1 \otimes 1,
\alpha_s \otimes 1 \}$ of degrees $0$ and $2$. More generally, one can
show that
\[
R \otimes_{R^{s_1}} R \otimes_{R^{s_2}} \dots
R \otimes_{R^{s_m}} R \otimes_R\RM = R \otimes_{R^{s_1}} R \otimes_{R^{s_2}} \dots
R \otimes_{R^{s_m}} \RM
\]
has graded basis $\alpha_{s_1}^{\e_1} \otimes \alpha_{s_2}^{\e_2}
\otimes \dots \otimes \alpha_{s_m}^{\e_m} \otimes 1$ where $(\e_a)_{a =
  1}^m$ is any tuple of zeroes and ones. In particular, its Poincar\'e
polynomial is $(1+q^2)^m$.

Recall that in the introduction we described the Borel isomorphism:
\[
H^\bullet(G/B) \cong R / (R^W_+).
\]
Notice that left multiplication by any invariant polynomial of
positive degree acts as zero on
\[
H^\bullet(BS(s_1, \dots, s_m)) = R \otimes_{R^{s_1}} R \otimes_{R^{s_2}} \dots
R \otimes_{R^{s_m}} R \otimes_R\RM.
\]
We conclude that $R \otimes_{R^{s_1}} \dots
R \otimes_{R^{s_m}} R \otimes_R\RM$ is a module over $R
/ (R^W_+)$. Geometrically, this corresponds to the the pull-back map
on cohomology 
\[
H^\bullet(G_\CM/B) \to H^\bullet(BS(s_1, \dots, s_m)) 
\]
discussed above.

We can now reformulate Theorem \ref{thm:s1} algebraically as follows:

\begin{thm}[Soergel \cite{So}] \label{thm:s2}
Let $D_w$ be any indecomposable $R/ (R^W_+)$-module direct
  summand of 
\[
H^\bullet(BS(s_1, \dots, s_m)) = R \otimes_{R^{s_1}} R \otimes_{R^{s_2}} \dots
R \otimes_{R^{s_m}} R \otimes_R\RM
\]
containing the element $1 \otimes 1 \otimes \dots \otimes 1$, where $w
= s_1 \dots s_m$ is a reduced expression for $w$. Then
$D_w$ is well-defined up to isomorphism (i.e. does not depend on the choice of reduced
expression) and $D_w \cong IH^\bullet(X_w)$.
\end{thm}

The modules $\{ D_w \; | \; w \in W \}$ are the (indecomposable)
\emph{Soergel modules}.

\begin{ex}
  We consider the case of $G = GL_3(\CM)$ in which case
\[
W = S_3 = \{
  \id, s_1, s_2, s_1s_2, s_2s_1, s_1s_2s_1 \}
\]
(we use the conventions of Example \ref{ex:bs}). In this case it turns
out that all Schubert varieties are smooth. Also, if $\ell(w) \le 2$
then any Bott-Samelson resolution is an isomorphism. We conclude
\begin{align*}
D_{\id}  = \RM  &\\
D_{s_1} = H^\bullet(BS(s_1)) = R \otimes_{R^{s_1}} \RM \quad & \quad 
D_{s_2} = R \otimes_{R^{s_2}} \RM
\\
D_{s_1s_2} = H^\bullet(BS(s_1, s_2)) = R \otimes_{R^{s_1}} R \otimes_{R^{s_2}}
\RM \quad & \quad D_{s_2s_1} = R \otimes_{R^{s_2}} R \otimes_{R^{s_1}}
\RM
\end{align*}
(A pleasant exercise for the reader is to verify that in all these
examples above $D_x$ is a cyclic (hence indecomposable) module over $R$. This is
not usually the case, and is related to the (rational) smoothness of the Schubert varieties in
question.)

The element $w_0 = s_1s_2s_1$ is more interesting. In this case the
Bott-Samelson resolution
\[
BS(s_1, s_2, s_1)  \to X_{w_0} = G/B
\]
is not an isomorphism. As previously discussed, the Poincar\'e
polynomial of
\begin{equation} \label{eq:gl3bs}
H^\bullet(BS(s_1, s_2, s_1) ) = R \otimes_{R^{s_1}} R \otimes_{R^{s_2}} R \otimes_{R^{s_1}}
\RM
\end{equation}
is $(1+q^2)^3$ whereas the Poincar\'e polynomial of
\begin{equation} \label{eq:gl3c}
IH^\bullet(X_{w_0}) = H^\bullet(G/B) = R/(R^W_+)
\end{equation}
is $(1+q^2)(1+q^2+q^4)$. In this case the reader may verify that
\eqref{eq:gl3c} is a summand of \eqref{eq:gl3bs}. In fact one has an
isomorphism of graded $R/(R^W_+)$-modules:
\[
R \otimes_{R^{s_1}} R \otimes_{R^{s_2}} R \otimes_{R^{s_1}}
\RM
= R/(R^W_+) \oplus (R\otimes_{R^s} \RM(-2)).
\]
Here $R\otimes_{R^s} \RM(-2)$ denotes the shift of $R\otimes_{R^s}
\RM$ in the grading such that its generator $1 \otimes 1$ occurs in
degree 2. This extra summand can be embedded into \eqref{eq:gl3bs} via
the map which sends \[f \otimes 1 \mapsto f \otimes \alpha_{s_2}
\otimes 1 \otimes 1 + f \otimes 1 \otimes \alpha_{s_2} \otimes 1\] for
$f \in R$.
\end{ex}

\begin{ex} \label{ex:w0}
If $w_0$ denotes the longest element of $W$ then $X_{w_0} =
G_{\CM}/B$, the (smooth) flag variety of $G$. In particular
\[
IH^\bullet(X_{w_0}) = H^\bullet(G_{\CM}/B) = R/(R^W_+)
\]
by the Borel isomorphism. Theorem \ref{thm:s2} asserts
that $R/(R^W_+)$ occurs as a direct summand of
\[
R \otimes_{R^{s_1}} R \otimes_{R^{s_2}} \otimes \dots
\otimes_{R^{s_m}} \RM
\]
for any reduced expression $w_0 = s_1\dots s_m$. This is by no means obvious! We have seen an
instance of this in the previous example.
\end{ex}

\begin{rmk} In this section we could have worked in the category of graded $R$-modules, rather than the category of graded $R/(R^W_+)$-modules, and it would change nothing. All the
$R$-modules in question will factor through $R/(R^W_+)$. In the next section, we will work with $R$-modules instead. \end{rmk}

We now discuss hard Lefschetz and the Hodge-Riemann bilinear
relations. Recall that our Borel subgroup $B \subset
G_\CM$ determines a set of simple roots $\Delta \subset \Phi
\subset (\Lie T)^*$ and simple coroots $\Delta^\vee \subset \Phi^\vee
\subset \Lie T$. Under the isomorphism
\[
H^2(G_\CM/B) \cong (\Lie T)^*
\]
the ample cone (i.e. the $\RM_{> 0}$-stable subset of $H^2(G_\CM/B)$ generated
by Chern classes of ample line bundles on $G_\CM/B$) is the cone of dominant
weights for $\Lie T$:
\begin{equation} \label{eq:ample cone}
(\Lie T)^*_+ := \{ \lambda \in (\Lie T)^* \; | \; \langle \lambda, \alpha^\vee \rangle
> 0 \text{ for all } \alpha^\vee \in \Delta^\vee \}.
\end{equation}
The hard Lefschetz theorem then asserts that left multiplication by
any $\lambda \in (\Lie T)^*_+$ satisfies the hard Lefschetz theorem on $D_w =
IH^\bullet(X_w)$. That is, for all $i \ge 0$, multiplication by
$\lambda^i$ induces an isomorphism
\[
\lambda^i : D_w^{\ell(w) -i} \simto D_w^{\ell(w) + i}.
\]

To discuss the Hodge-Riemann relations we need to make the Poincar\'e
pairing $\langle -, - \rangle_{\Poinc}$ explicit for $D_w$. We first
discuss the Poincar\'e form on
$H^\bullet(BS(s_1, \dots, s_m))$. Recall that for any oriented manifold $M$
the Poincar\'e form in de Rham cohomology is given by
\[
\langle \a, \b \rangle = \int_M \alpha \wedge \b.
\]
We imitate this algebraically as follows. By the discussion after
Proposition \ref{bs},
the degree $2m$ component of
\[
H^\bullet(BS(s_1, \dots, s_m)) = R \otimes_{R^{s_1}} R \otimes_{R^{s_2}} \dots
R \otimes_{R^{s_m}} \RM
\]
is one-dimensional and is spanned by the vector $c_{\top} := \alpha_{s_1} \otimes
\alpha_{s_2} \otimes \dots
\otimes \alpha_{s_m} \otimes 1$. We can define a bilinear form
$\langle -, - \rangle$ on $R \otimes_{R^{s_1}} R \otimes_{R^{s_2}} \dots
R \otimes_{R^{s_m}} \RM$ via
\[
\langle f, g \rangle = \Tr(fg)
\]
where $fg$ denotes the term-wise multiplication, and $\Tr$ is the
functional which returns the coefficient of $c_{\top}$. Then $\langle
-, - \rangle$ is a non-degenerate symmetric form which agrees up 
to a positive scalar with the intersection form on $H^\bullet(BS(s_1,
\dots, s_m))$. 

Now recall that $D_w$ is obtained as summand of $R \otimes_{R^{s_1}} R \otimes_{R^{s_2}} \dots
R \otimes_{R^{s_m}} \RM$, for a reduced expression of $w$. Fixing such an inclusion we obtain a form on
$D_w$ via restriction of the form $\langle -, - \rangle$. In fact,
this form is well-defined (i.e. depends neither on the choice of
reduced expression nor embedding) up to a positive scalar. One can
show that this form agrees with the Poincar\'e pairing on $D_w =
IH^\bullet(X_w)$ up to a positive scalar. The Hodge-Riemann bilinear
relations then hold for $D_w$ with respect to this form and left
multiplication by any $\lambda \in (\Lie T)^*_+$.

\section{Soergel modules for arbitrary Coxeter systems} \label{arbitrary}

Now let $(W,S)$ denote an arbitrary Coxeter system. That is, $W$ is a
group with a distinguished set of generators $S$ and a presentation
\[
W = \langle s \in S \; | \; (st)^{m_{st}} = \id \rangle
\]
such that $m_{ss} = 1$
and $m_{st} = m_{ts} \in \{ 2, 3, 4, \dots, \infty \}$ for all $s
\ne t$. (We interpret $(st)^{\infty} = \id$ as there being no
relation). As we discussed above, the Weyl groups of compact Lie
groups are Coxeter groups. In the 1930's Coxeter proved that the
finite reflection groups are exactly the finite Coxeter groups, and
achieved in this way a classification. As well as the finite
reflection groups arising in Lie theory (of types $A$, \dots, $G$) one
has the symmetries of the regular $n$-gon (a dihedral group of type
$I_2(n)$) for $n \ne 3, 4, 6$, the symmetries of the icosahedron (a group of type
$H_3$) and the symmetries of a regular polytope in $\RM^4$ with 600
sides (a group of type $H_4$).

It was realized later (by Coxeter, Tits, \dots) that Coxeter groups form an interesting
class of groups whether or not they are finite. They encompass groups
generated by affine reflections in euclidean space (affine Weyl
groups), certain hyperbolic reflection groups etc. One can treat these
groups in a uniform way thanks to the existence of their \emph{geometric
representation}. Let $\hg = \bigoplus_{s \in S} \RM \alpha_s^\vee$ for formal symbols $\alpha_s^\vee$,
and define a form on $\hg$ via
\[
(\alpha_s^\vee, \alpha_t^\vee) = -\cos(\pi/m_{st}).
\]
Although this form is positive definite if and
only if $W$ is finite, one can still imagine  that each $\alpha_s^\vee$ has length 1 and the angle between
$\alpha_s^\vee$ and $\alpha_t^\vee$ for $s \ne t$ is
$(m_{st}-1)\pi/m_{st}$. 
It is not difficult to verify (see \cite[V.4.1]{Bo} or \cite[5.3]{Hu}) that the
assignment
\[
s(v) := v - 2(v,\alpha_s^\vee)\alpha_s^\vee
\]
defines a representation of $W$ on $\hg$. In fact it is faithful
(\cite[V.4.4.2]{Bo} or
\cite[Corollary 5.4]{Hu})

If $W$ happens to be the Weyl group of our $T \subset G$ from the
introduction then (by rescaling the
coroots so that they all have length 1 with respect to a $W$-invariant
form) one may construct a $W$-equivariant isomorphism
\[
\Lie T \cong \hg.
\]
Hence one can think of this setup as providing the action of $W$
on the Lie algebra of a maximal torus, even though the corresponding
Lie group might not exist!

The main point of the previous section is that one may describe the
intersection cohomology, Poincar\'e pairing and ample cone entirely
algebraically, using only $\hg$, its basis and its $W$-action. That is, let us
(re)define $R = S(\hg^*)$ to be the symmetric algebra on $\hg^*$
(alias the regular functions on $\hg$), graded with $\deg \hg^* =
2$. Then $W$ acts on $R$ via graded algebra automorphisms. Imitating
the constructions of the previous section one obtains graded $R$-modules
$D_w$ (well-defined up to isomorphism), the only difference being that
we work in the category of $R$-modules rather than $R/(R^W_+)$-modules.\footnote{Although all the $R$-modules will factor through $R/(R^W_+)$, we prefer the ring $R$ for philosophical reasons. When $W$ is infinite, the ring $R/(R^W_+)$
is infinite-dimensional, as $R^W$ has the ``wrong'' transcendence
degree, and the Chevalley theorem does not hold. The ring $R$ behaves
in a uniform way for all Coxeter groups, while the quotient ring
$R/(R^W_+)$ does not. } 
We call the modules $D_w$ the (indecomposable) \emph{Soergel modules}.
As in the Weyl group case,
the modules $D_w$ are finite dimensional
over $\RM$ and are equipped with
non-degenerate ``Poincar\'e pairings'':
\[
\langle -, - \rangle : D_w^i \times D_w^{2\ell(w) - i} \to \RM.
\]

Our main theorem is that these modules $D_w$ ``look
like the intersection cohomology of a Schubert variety''. Consider the ``ample cone'':
\[
\hg^*_+ := \{ \lambda \in \hg^* \; | \; \langle \lambda, \alpha_s^\vee
\rangle > 0 \text{ for all } s \in S \}.
\]

\begin{thm}[\cite{EW}] \label{thm:ew}
For any $w \in W$, let $D_w$ be as above.
  \begin{enumerate}
  \item (Hard Lefschetz theorem) For any $i \le \ell(w)$, left
    multiplication by $\lambda^i$ for any $\l \in \hg^*_+$ gives an isomorphism
\[
\lambda^{\ell(w) -i} : D_w^{i} \simto D_w^{2\ell(w) - i}
\]
\item (Hodge-Riemann bilinear relations) For any $i \le \ell(w)$ and
  $\lambda \in \hg^*_+$ the restriction of the form
\[
(f, g) := \langle f, \lambda^{\ell(w)-i} g \rangle
\]
on $D_w^{i}$ to $P^{i} = \ker \lambda^{\ell(w)-i+1}
\subset D_w^{i}$ is $(-1)^{i/2}$-definite.
  \end{enumerate}
\end{thm}

Some remarks:
\begin{enumerate}
\item The graded modules $D_w$ are zero in odd-degree (as is immediate
  from their definition as a summand of $R \otimes_{R^{s_1}} \dots
  \otimes_{R^{s_m}} \RM$) and so the sign $(-1)^{i/2}$ makes sense.
\item The motivation behind establishing the above theorem is a
  conjecture made by Soergel in \cite[Vermutung 1.13]{Soe3}. In fact, the above
  theorem forms part of a complicated inductive proof of Soergel's conjecture.
Soergel was led to his
  conjecture as an algebraic means of understanding the Kazhdan-Lusztig
  basis of the Hecke algebra and the Kazhdan-Lusztig conjecture on
  characters of simple highest weight modules over complex semi-simple
  Lie algebras. The definition of the Kazhdan-Lusztig basis and the
  statement of the Kazhdan-Lusztig conjecture is ``elementary''
but, prior to the above results, needed powerful tools
  from algebraic geometry (e.g. Deligne's proof of the Weil
  conjectures) for its resolution. Because of this reliance on
  algebraic geometry, these methods break down for arbitrary Coxeter
  systems, for which no flag variety exists. In some sense the above
  theorem is interesting because it provides a ``geometry'' for
  Kazhdan-Lusztig theory for Coxeter groups which do not come from
  Lie groups or generalizations (affine, Kac-Moody, \dots) thereof. This was Soergel's aim in formulating his conjecture.
\item Our proof is inspired by the beautiful work of de Cataldo and
  Migliorini \cite{dCM,dCM2}, which
  proves the decomposition theorem using only classical Hodge
  theory.
\item The idea of considering the ``intersection cohomology'' of a
  Schubert variety associated to any element in a Coxeter group has
  also been pursued by Dyer \cite{Dyer1, Dyer} and Fiebig
  \cite{Fiebig}. There is also a
closely related theory non-rational polytopes (where the associated
toric variety is missing) \cite{BL,Karu,BKBF}.
\item In Example \ref{ex:w0} we saw that if $W$ is a Weyl group then an
  important example of a Soergel module is
\[
D_{w_0} \cong R/(R^W_+).
\]
In fact this isomorphism holds for any finite Coxeter group $W$
with longest element $w_0$. The ``coinvariant''\footnote{W. Soergel
  pointed out that this is a bad name, as it has
  nothing whatsoever to do with coinvariants.} algebra $R/(R^W_+)$
has been studied by many authors from many points of view. However
even in this basic example it seems to be difficult to check the
hard Lefschetz theorem or Hodge-Riemann bilinear relations
directly. In the next section we will do this by hand when $W$ is
a dihedral group.
\item In \cite{EW} we work with $\hg$ a slightly larger representation
  containing the geometric representation. We do this for technical
  reasons (to ensure that the category of Soergel
  bimodules is well-behaved). However, one can deduce
  Theorem \ref{thm:ew} from the results of \cite{EW}. The idea of
  using the results for the slightly larger representation to deduce
  results for the geometric representation goes back to Libedinsky \cite{Lib}.
\item (For the experts.) In \cite{EW} we prove the results above for certain $R$-modules $\overline{B_w}$, whose definition differs subtly from that of $D_w$. However, given that $\overline{B_w}$ is indecomposable as an $R$-module, one can show easily that $\overline{B_w}$ and $D_w$ are isomorphic. This will be explained elsewhere.
\end{enumerate}

\section{The flag variety of a dihedral group} \label{dihedral}

In this final section we amuse ourselves with the coinvariant ring of
a finite dihedral group. We check the hard Lefschetz property and
Hodge-Riemann bilinear relations directly.

\subsection{Gau\ss's $q$-numbers}
We start by recalling Gau\ss's $q$-numbers. By definition
\[
[n] := q^{-n+1} + q^{-n+3} + \dots + q^{n-3} + q^{n-1} = \frac{q^n-q^{-n}}{q-q^{-1}} \in \ZM[q^{\pm 1}].
\]
Many identities between numbers can be lifted to identities between
$q$-numbers. We will need
\begin{align}
  \label{eq:1}
 [2] [n] &= [n+1] + [n-1] \\\label{eq:2}
[n]^2 &= [2n-1] + [2n-3] + \dots + [1].\\\label{eq:3}
[n][n+1] &= [2n] + [2n-2] + \dots + [2].
\end{align}
For the representation theorist, $[n]$ is the character of the simple
$\mathfrak{sl}_2(\mathbb{C})$-module of dimension $n$, and the relations above
are instances of the Clebsch-Gordan formula.

If $\z = e^{2\pi  i/2m} \in \CM$ then we can specialize $q = \z$ to
obtain algebraic integers $[n]_\z \in \RM$. Because $\z^m = -1$ we have
\begin{align}
[m]_\z = 0,  \quad [i]_\z = [m-i]_\z, \quad
[i+m]_\z = -[i]_\z.
\end{align}

Because $\z^n$ has positive imaginary part for $n<m$, it is clear that
\begin{equation} \label{pos}
\text{$[n]_\z$ is positive for $0 < n < m$.}
\end{equation}
We use this positivity in a crucial way below. Had we foolishly chosen $\z$ to be a primitive $2m^{\textrm{th}}$ root of unity with non-maximal real part, \eqref{pos} would fail.

\subsection{The reflection representation of a dihedral group}
Now let $W$ be a finite dihedral group of order $2m$. That is $S = \{
s_1, s_2 \}$ and
\[
W = \langle s_1, s_2 \; | \; s_1^2 = s_2^2 = (s_1s_2)^m = \id \rangle.
\]
Let $\hg = \RM \alpha_1^\vee \oplus \RM {\alpha_2^\vee}$ be the
geometric representation of $(W,S)$, as in \S\ref{arbitrary}. Because
$W$ is finite the form $(-,-)$ on $\hg$ is non-degenerate.  We
define simple roots $\alpha_1, \alpha_2 \in \hg^*$ by
$\alpha_1 = 2(\alpha_1^\vee, -)$ and $\alpha_2 = 2(\alpha^\vee_2, -)$. Then
the ``Cartan'' matrix is
\begin{equation} \label{Cartan}
(\langle \alpha^\vee_i, \alpha_j \rangle)_{i,j\in \{1,2 \}} = 
\left ( \begin{matrix} 2 & -\varphi \\ -\varphi & 2 \end{matrix}
\right )
\end{equation}
where $\varphi = 2 \cos(\pi/m)$. 
Note that $\varphi = \z + \z^{-1}$ where $\z = e^{2\pi
  i/2m} \in \CM$. Hence $\varphi = [2]_\z$ in the notation of the
previous section. In particular it is an algebraic integer.


\begin{ex}
Throughout we will use the first non-Weyl-group case $m = 5$ to
illustrate what is going on. In this case $[2]_\z = [3]_\z$
and the relation $[2]^2 = [3] + [1]$ gives
$\varphi^2 = \varphi + 1$. Thus $\varphi$ is the golden ratio.
\end{ex}

For all $v \in \hg^*$ we have
\[
s_1(v) = v - \langle v, \alpha_1^\vee \rangle \alpha_1 \quad \text{and}
\quad 
s_2(v) = v - \langle v, \alpha_2^\vee \rangle \alpha_2.
\] It is a pleasant exercise for the reader to verify that the set $\Phi
=  W \cdot \{ \alpha_1, \alpha_2 \}$ gives something like a root
system in $\hg^*$. We have $\Phi = \Phi^+ \cup -\Phi^+$ where
\begin{equation} \label{posroots}
\Phi^+ = \{ [i]_\z \alpha_1 + [i-1]_\z\alpha_2 \; | \; 1 \le i \le m \}.
\end{equation}

\begin{ex} For $m = 5$ one can picture the ``positive roots'' $\Phi^+$
  as follows:
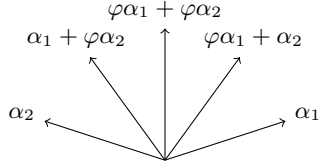
\[
\begin{tikzpicture}
\node (as) at (18:2cm) {\small $\alpha_1$};
\node (a1) at (54:2cm) {\small $\varphi \alpha_1 + \alpha_2 $};
\node (a2) at (90:2cm) {\small $\varphi \alpha_1 + \varphi \alpha_2$};
\node (a3) at (126:2cm) {\small $\alpha_1 + \varphi \alpha_2$};
\node (at) at (162:2cm) {\small $\alpha_2$};
\draw[->] (0,0) to (as);
\draw[->] (0,0) to (a1);
\draw[->] (0,0) to (a2);
\draw[->] (0,0) to (a3);
\draw[->] (0,0) to (at);
\end{tikzpicture}
\]
\end{ex}

Let $T := \bigcup wSw^{-1}$. Then $T$ are precisely the elements of
$W$ which act as reflections on $\hg$ (and $\hg^*$). One has a bijection
\[
T \simto \Phi^+: t \mapsto \alpha_t
\]
such that $t(\alpha_t) = -\alpha_t$ for all $t \in T$.

\subsection{Schubert calculus}

In the following we describe Schubert calculus for the coinvariant
ring. Most of what we say here is valid for any finite Coxeter
group. A good reference for the unproved statements below is
\cite{Hiller}.

Let $R$ denote the symmetric algebra on $\hg^*$ and $H$ the
coinvariant algebra
\[
H := R/(R^W_+).
\]
For each $s \in S$ consider the divided difference operator
\[
\partial_s(f) = \frac{f - s(f)}{\alpha_s}.
\]
Then $\partial_s$ preserves $R$ and decreases degrees by
$2$. Given $x \in W$ we define
\[
\partial_x = \partial_{s_1} \dots \partial_{s_m}
\]
where $x = s_1 \dots s_m$ is a reduced expression for $x$. The operators $\partial_s$ satisfy the braid relations, and therefore
$\partial_x$ is well-defined. The operators $\partial_x$ kill
invariant polynomials and hence commute with multiplication by
invariants. In particular they preserve the ideal $(R^W_+)$
and induce operators on $H$.

Let $\pi := \Pi_{\alpha \in \Phi^+} \alpha$ denote the product of the
positive roots. For any $x \in W$ define $Y_x \in H$ as the image of $\partial_x
(\pi)$ in $H$. Because $\pi$ has degree $2\ell(w_0)$, $Y_x$ has degree
$\deg Y_x = 2(\ell(w_0) - \ell(x))$.

\begin{thm} The elements $\{ Y_x \; | \; x \in W \}$ give a basis for
  $H$. \end{thm}

This basis is called the \emph{Schubert basis}. When $W$ is a Weyl
group each $Y_x$ maps under the Borel isomorphism to the
fundamental class of a Schubert variety \cite{BGG}.

We can define a bilinear form $ \langle -, - \rangle$ on $H$
as follows:
\[
\langle f, g \rangle := \frac{1}{2m}\partial_{w_0}(fg)
\]
Then for all $x, z \in W$ one has:
\begin{equation} \label{eq:pairing}
\langle Y_x, Y_z \rangle = \delta_{w_0, x^{-1}z}.
\end{equation}
In particular $\langle -, - \rangle$ is a non-degenerate form on $H$.

The following ``Chevalley'' formula describes the action of an element $f \in \hg^*$
in the basis $\{ Y_x \}$:
\begin{equation} \label{eq:chevalley}
f \cdot Y_x = \sum_{t \in T \atop \ell(tx) = \ell(x) -1} \langle f ,
\alpha_t^\vee \rangle Y_{tx}
\end{equation}
\begin{ex}
Figure 1 depicts the case $m = 5$. Each edge is labelled with the
coroot which, when paired against $f$, gives the scalar coefficient
that describes the action of $f$. Using \eqref{posroots} the reader can guess what
the picture looks like for general $m$.
\end{ex}

\begin{figure}\caption{The Chevalley formula for the dihedral group with $m = 5$:}
\begin{center}
\begin{tikzpicture}
\node (id) at (90:4cm) {$Y_{\id}$};
\node (2) at (54:4cm) {$Y_{s_2}$};
\node (21) at (18:4cm) {$Y_{s_2s_1}$};
\node (212) at (342:4cm) {$Y_{s_2s_1s_2}$};
\node (2121) at (306:4cm) {$Y_{s_2s_1s_2s_1}$};
\node (w0) at (270:4cm) {$Y_{w_0}$};
\node (1) at (126:4cm) {$Y_{s_1}$};
\node (12) at (162:4cm) {$Y_{s_1s_2}$};
\node (121) at (198:4cm) {$Y_{s_1s_2s_1}$};
\node (1212) at (234:4cm) {$Y_{s_1s_2s_1s_2}$};

\draw (2) [->] -- node[sloped,above] {$\alpha^\vee_2$}  (id);
\draw (12) [->] -- node[sloped,above,near start] {$\alpha^\vee_1$} (2);
\draw (121) [->] -- node[sloped,above,near start] {$\alpha^\vee_1$} (21);
\draw (1212) [->] -- node[sloped,above,near end] {$\alpha^\vee_1$} (212);
\draw (w0) [->] -- node[sloped,below] {$\alpha^\vee_1$} (2121);

\draw (1) [->] -- node[sloped,above] {$\alpha^\vee_1$}  (id);
\draw (21) [->] -- node[sloped,above,near start] {$\alpha^\vee_2$} (1);
\draw (212) [->] -- node[sloped,above,near start] {$\alpha^\vee_2$} (12);
\draw (2121) [->]  -- node[sloped,above,near end] {$\alpha^\vee_2$} (121);
\draw (w0) [->] -- node[sloped,below] {$\alpha^\vee_2$} (1212);

\draw (2121) [->] -- node[right] {$\varphi \alpha^\vee_1 + \alpha^\vee_2$} (212);
\draw (212) [->] --node[right] {$\varphi \alpha^\vee_1 + \varphi \alpha^\vee_2$}  (21);
\draw (21) [->] --node[right] {$\alpha^\vee_1 + \varphi \alpha^\vee_2$}  (2);

\draw (1212) [->] -- node[left] {$\varphi \alpha^\vee_2 + \alpha^\vee_1$} (121);
\draw (121) [->] --node[left] {$\varphi \alpha^\vee_2 + \varphi \alpha^\vee_1$}  (12);
\draw (12) [->] --node[left] {$\alpha^\vee_2 + \varphi \alpha^\vee_1$}  (1);
\end{tikzpicture}
\end{center}
\end{figure}
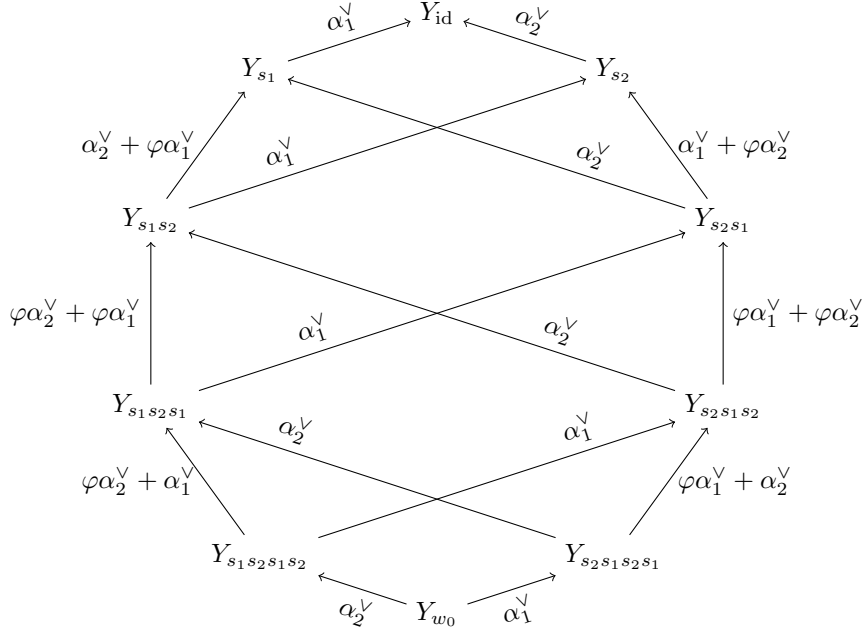

\begin{prop} Suppose that $\l \in \hg^*$ is such that $\langle \alpha_i^\vee,
  \l \rangle > 0$ for $i = 1,2$. Then multiplication by $\l$ on
  $H$ satisfies the hard Lefschetz 
  theorem, and the Hodge-Riemann bilinear relations hold.\end{prop}

\begin{proof} It is immediate from \eqref{eq:chevalley} that if $\l$ is as
  in the proposition and if $x \ne \id$ then $\l Y_x$ is a sum of various
  $Y_{z}$ with strictly positive coefficients (two terms occur if
  $\ell(x) < m-1$ and one term occurs if $\ell(x) = m-1$). Hence $\l^m
  Y_{w_0}$ is a strictly positive constant times $Y_{\id}$. In
  particular $\l^m : H^0 = \RM Y_{w_0}  \to H^{2m} = \RM Y_{\id}$ is an
  isomorphism. By \eqref{eq:pairing} we have
\[
\langle Y_{w_0}, \l^m Y_{w_0} \rangle > 0
\]
and hence the Lefschetz form is positive definite on $H^0$. 

We now fix $1 \le i < m-1$ and consider multiplication by $f \in \hg^*$ as
a map $H^{2i} \to H^{2i + 2}$. The following diagram depicts the
effect in the Schubert basis:
\begin{equation} \label{middle}
\begin{array}{c}
\begin{tikzpicture}
\node (ul) at (0,2) {$Y_a$};
\node (ur) at (3,2) {$Y_b$};
\node (ll) at (0,0) {$Y_{s_1b}$};
\node (lr) at (3,0) {$Y_{s_2a}$};

\draw (ll) [->] --node[left] {$[i]_\z\alpha^\vee_1 + [i+1]_\z\alpha^\vee_2$}  (ul);
\draw (lr) [->] --node[right] {$[i+1]_\z\alpha^\vee_1 + [i]_\z\alpha^\vee_2$}  (ur);
\draw (ll) [->] --node[sloped, near start,above] {$\alpha^\vee_1$} (ur);
\draw (lr) [->] --node[sloped,near start,above] {$\alpha^\vee_2$} (ul);

\end{tikzpicture}
\end{array}
\end{equation}
where $a$ and $b$ (resp. $s_2a$ and $s_1b$) are the unique elements of
length $\ell(w_0) - i - 1$ (resp. $\ell(w_0) - i$). Remember that $\a_i^\vee$ here represents the scalar $\langle f ,
\alpha_i^\vee \rangle$.
We now calculate the determinant:
\begin{gather*}\det \left ( \begin{matrix} [i]_\z\alpha^\vee_1 + [i+1]_\z\alpha^\vee_2 & \alpha^\vee_2\\
\alpha^\vee_1 & [i+1]_\z\alpha^\vee_1 +  [i]_\z\alpha^\vee_2 \end{matrix} \right ) =\\
 = [i]_\z[i+1]_\z (\alpha^\vee_1)^2 + ([i]^2_\z + [i+1]^2_\z -
1)\alpha^\vee_1\alpha^\vee_2 + [i]_\z[i+1]_\z(\alpha^\vee_2)^2 \\
= [i]_\z[i+1]_\z(\alpha^\vee_1)^2 + [2]_\z[i]_\z[i+1]_\z \alpha^\vee_1\alpha^\vee_2 + [i]_\z[i+1]_\z(\alpha^\vee_2)^2
\end{gather*}
(using \eqref{eq:1}, \eqref{eq:2} and \eqref{eq:3}). All $q$-numbers appearing here are positive by \eqref{pos}.

If $\lambda$ is as in the proposition, then the determinant of multiplication by $\lambda$ is positive. So $\lambda$ gives an isomorphism $H^{2i} \simto H^{2i+2}$ for each $1\le i \le
m-2$, and $\l^{m-2}$ gives an isomorphism $H^2 \simto
H^{2m-2}$. Therefore the hard Lefschetz theorem holds for $\lambda$,
with primitive classes occurring only in degrees $0$ and $2$.

It remains to check the Hodge-Riemann bilinear relations. We have already seen that the Lefschetz form on $H^0$ is positive definite. We need to know that the restriction of the Lefschetz
form on $H^2$ to $\ker \l^{m-1}$ is negative definite. Now $(\l Y_{w_0}, \l Y_{w_0}) = (Y_{w_0}, Y_{w_0}) > 0$, and if $\g \in H^2$ denotes a generator for $\ker \l^{m-1}$ then $( \l
Y_{w_0}, \g) = \langle \l Y_{w_0}, \l^{m-2} \g \rangle = \langle Y_{w_0}, \l^{m-1} \g \rangle = 0$. Hence the Hodge-Riemann relations hold if and only if the signature of the Lefschetz
form on $H^2$ is zero.

From the definition of the Lefschetz form, it is immediate that $\l : H^{2i} \to H^{2i+2}$ is an isometry with respect to the Lefschetz forms, so long as $2 \le 2i \le m-2$. Thus when $m$
is even (resp. odd) it is enough to show that the signature of the
Lefschetz form is zero on $H^m$ (resp. $H^{m-1}$).

Suppose $m$ is even. The
Lefschetz form on the middle dimension $H^m$ is the same as the pairing. By
\eqref{eq:pairing} this form has Gram matrix
\[
\left ( \begin{matrix} 0 & 1 \\ 1 & 0 \end{matrix} \right )
\]
which has signature 0.

Suppose $m=2k+1$ is odd; we check the signature of the Lefschetz form on
$H^{m-1}$. We are reduced to studying \eqref{middle} with $\ell(a)
= \ell(b) = k$ and $\ell(s_2a) =\ell(s_1b) =
k+1$. We
see by \eqref{eq:pairing} that $Y_{s_1b}, Y_{s_2a}$ is a basis dual to
$Y_b$, $Y_a$. We get that the
Lefschetz form on $H^{m-1}$ is given by
\[ \left ( \begin{matrix} \alpha^\vee_1 & [k+1]_\z\alpha^\vee_1 + [k]_\z\alpha^\vee_2 \\
[k]_\z\alpha^\vee_1 + [k+1]_\z\alpha^\vee_2 &
\alpha^\vee_2 \end{matrix} \right ), \]
and $[k]=[k+1]$ is positive. For any $\l$ as in the proposition, this
is a symmetric matrix with
strictly positive entries and negative determinant (by our 
calculation above).  Hence its signature is zero and the Hodge-Riemann
relations are satisfied as claimed.
\end{proof}

\def\cprime{$'$}

\end{document}